\documentclass[preprint,11pt]{elsarticle}
\usepackage[T1]{fontenc}
\usepackage{booktabs}
\usepackage{}
\usepackage{amsmath,amssymb,amsthm}
\usepackage{enumitem}
\usepackage{graphicx}
\usepackage{color}
\usepackage{soul}
\usepackage{color, xcolor}
\usepackage[
  colorlinks=true,
  linkcolor=blue,
  citecolor=green,
  urlcolor=red
]{hyperref}
\usepackage[all]{xy}
\makeatletter
\def\ps@pprintTitle{%
  \let\@oddhead\@empty
  \let\@evenhead\@empty
  \def\@oddfoot{\reset@font\hfil\thepage\hfil}
  \let\@evenfoot\@oddfoot}
\makeatother

\newcommand{\R}{\mathbb{R}}
\newcommand{\Q}{\mathbb Q}

\newcommand{\N}{\mathbb{N}}
\newcommand{\Z}{\mathbb{Z}}

\newcommand{\SG}{\mathbb{SG}}
\newcommand{\QG}{\mathbb{QG}}
\newcommand{\FCSG}{\mathbb{FCSG}}
\newcommand{\FCQG}{\mathbb{FCQG}}
\newcommand{\SCSG}{\mathbb{SCSG}}
\newcommand{\SCQG}{\mathbb{SCQG}}

\DeclareMathOperator{\Reg}{Reg}
\DeclareMathOperator{\Tych}{Tych}
\newcommand{\nw}{\operatorname{nw}}
\newcommand{\cl}{\operatorname{cl}}
\newcommand{\Int}{\operatorname{Int}}

\newcommand{\eps}{\varepsilon}

\newtheorem{defin}{Definition}[section]

\newtheorem{theorem}[defin]{Theorem}
\newtheorem{lemma}[defin]{Lemma}
\newtheorem{proposition}[defin]{Proposition}
\newtheorem{corollary}[defin]{Corollary}
\newtheorem{remark}[defin]{Remark}
\newtheorem{qu}{Question}

\newenvironment{question}{\begin{qu}\rm}{\end{qu}}

\begin{document}

\begin{frontmatter}

\title{Answer some Tkachenko's questions in factorizable semitopological groups}

\author{Li-Hong Xie}
\ead{yunli198282@126.com}


\address{School of Mathematics and Computational Science, Wuyi University, Jiangmen, Guangdong, 529000, P.R. China}

\begin{abstract}
We construct a Hausdorff quasitopological abelian group that is
second-countable and semiregular, with a bounded continuous function
that does not factor
through any regular first-countable semitopological group. Its regular
and Tychonoff reflections coincide, have a countable network, and fail
even factorizability through arbitrary first-countable semitopological
groups. This gives negative answers to \cite[Problems~6.1, 6.2, 6.4, and~6.5]{Tk2026}
of Tkachenko. The construction starts with a countable cone group whose
common separation reflection is not Fr\'echet--Urysohn, answering
\cite[Problems~6.8 and~6.9]{Tk2026} negatively. 
In contrast, symmetrization gives affirmative answers to \cite[Problems~6.6
and~6.7]{Tk2026} without separation assumptions. We also show that every
countable semitopological group admits simultaneous factorizations of
countable families through second-countable Tychonoff semitopological groups.
\end{abstract}


\end{frontmatter}

\section{Introduction}
The notion of an $\R$-factorizable topological group is a classical topic in
topological algebra. A topological group $G$ is called \emph{$\R$-factorizable}
if every continuous real-valued function $f:G\to\R$ admits a factorization $ f=h\circ p,$
where $p:G\to H$ is a continuous surjective homomorphism onto a
second-countable Hausdorff topological group $H$ and $h:H\to\R$ is continuous.
This concept, together with its variants such as $m$-factorizability, is
systematically developed in \cite{AT}. In recent years, factorization
properties have been extensively studied beyond the category of topological
groups. When multiplication is only separately continuous, the choice of the
category of homomorphic images becomes essential, and several non-equivalent
notions of $\R$-factorizability may arise.

A \emph{semitopological group} is a group whose left and right translations
are continuous; a \emph{quasitopological group} is a semitopological group with
continuous inversion; and a \emph{paratopological group} is a group whose
multiplication is jointly continuous. No separation axiom is assumed in these
definitions. Peng and Wang \cite{PW} studied
$\mathcal{PR}$ ($\mathcal{P}m$)-factorizable semitopological groups,
established internal characterizations of certain continuous homomorphic
images, and investigated the preservation of factorization properties under
open and closed homomorphisms. To compare the resulting factorization properties in a uniform
way, Tkachenko \cite{Tk2026} introduced the notation
$(\mathbb X,\mathbb O)$-factorizability, where $\mathbb X$ is a class
of target spaces and $\mathbb O$ is a class of topological-algebraic objects. A group $G$ with a topology is \emph{$(\mathbb X,\mathbb O)$-factorizability} \cite[Definition 1.1]{Tk2026} if for every continuous mapping $f : G \to
X$ to a space $X \in \mathbb X$, there exists a continuous surjective homomorphism $p: G \to H$, where $H\in \mathbb O$,
such that $f = h \circ p$, for some continuous mapping $h : H \to X$. Using this framework, he
proved the absoluteness of $\R$-factorizability and $\mathcal M$-factorizability
in the categories of topological, paratopological, quasitopological, and
semitopological groups; see \cite[Theorems 4.13, 4.16, 4.17, 4.22, 4.23,
4.27]{Tk2026}.
In \cite[Section 6]{Tk2026}, Tkachenko posed a series of open problems
concerning the behaviour of factorization properties under separation
reflections and the relationship between factorization in different
categories. We use $\FCSG$ and $\SCSG$ for the
classes of first-countable and second-countable semitopological groups,
respectively; replacing $\SG$ by $\QG$ specifies quasitopological groups.
A subscript $r$ or $t$ specifies regular or Tychonoff members of a class,
respectively. Throughout this
paper, \emph{regular} means regular and $T_1$. The regular reflection
$\Reg(G)$ and the Tychonoff reflection $\Tych(G)$ are used in the sense
of~\cite{Tk2014}. Their canonical maps are continuous surjective
homomorphisms and have a universal property for \emph{all continuous maps}
to regular or Tychonoff spaces, respectively. This includes real-valued
functions, not only homomorphisms. 

 The questions solved in the present paper are the following.

\begin{question}\cite[Problem 6.1]{Tk2026}\label{Q1}
Let $G$ be an $(\R,\SCSG)$-factorizable semitopological group. Is the regular reflection
    $\Reg(G)$ or the Tychonoff reflection $\Tych(G)$ of $G$ again
    $(\R,\SCSG)$-factorizable?
\end{question}

\begin{question}\cite[Problem 6.2]{Tk2026}\label{Q2}
 Let $G$ be an
    $(\R,\FCSG)$-factorizable semitopological group. Is $\Reg(G)$ or
    $\Tych(G)$ again $(\R,\FCSG)$-factorizable?
\end{question}

\begin{question}\cite[Problem 6.4]{Tk2026}\label{Q3}
 Let $G$ be an
    $(\R,\SCSG)$-factorizable semitopological group. Is $G$ necessarily
    $(\R,\SCSG_r)$-factorizable?
\end{question}

\begin{question}\cite[Problem 6.5]{Tk2026}\label{Q4}
Does the
    $(\R,\FCSG)$-factorizability of a semitopological group $G$ imply its
    $(\R,\FCSG_r)$-factorizability or even its
    $(\R,\FCSG_t)$-factorizability?
\end{question}

\begin{question}\cite[Problem 6.6]{Tk2026}\label{Q5}
Does the
    $(\R,\FCSG)$-factorizability of a quasitopological group $G$ imply its
    $(\R,\FCQG)$-factorizability? What if $G$ is regular?
\end{question}

\begin{question}\cite[Problem 6.7]{Tk2026}\label{Q6}
Does the
    $(\R,\SCSG)$-factorizability of a quasitopological group $G$ imply its
    $(\R,\SCQG)$-factorizability? What if $G$ is regular?
\end{question}

\begin{question}\cite[Problem 6.8]{Tk2026}\label{Q7}
Let $G$ be a second-countable
    semitopological (or quasitopological) group. Is the regular reflection
    $\Reg(G)$ or the Tychonoff reflection $\Tych(G)$ of $G$ second-countable?
\end{question}

\begin{question}\cite[Problem 6.9]{Tk2026}\label{Q8}
Let $G$ be a first-countable
    semitopological (or quasitopological) group. Is $\Reg(G)$ or $\Tych(G)$
    first-countable?
\end{question}

In this paper we answer the above Questions as follows.

In Section \ref{sec2}, Questions \ref{Q7} and \ref{Q8} are answered negatively. We construct a
countable, second-countable, Hausdorff, semiregular quasitopological abelian
group $G$ such that its regular reflection $\Reg(G)$ and Tychonoff reflection
$\Tych(G)$ coincide with a countable, zero-dimensional quasitopological group
that is not Fr\'echet--Urysohn (see Theorem \ref{thm:main-example}). Consequently, neither first countability nor
second countability is preserved by the regular or Tychonoff reflection, even
for countable Hausdorff quasitopological abelian groups. The construction is
based on a rational version of a topology on $\R^2$; the additional conclusion
about the separation reflections is proved by using continuous functions
supported on small Euclidean balls outside a fixed cone. In Section \ref{sec2}, Questions 6.1, 6.2, 6.4, and 6.5 are answered negatively. We construct
a Hausdorff quasitopological abelian group $H$, which is second-countable and
semiregular, and $(\R,\SCSG)$-factorizable but not
$(\R,\FCSG_r)$-factorizable, whereas neither $\Reg(H)$ nor $\Tych(H)$ is
$(\R,\FCSG)$-factorizable (see Theorem \ref{thm:main-factorization}).
This gives negative answers to Questions \ref{Q1}, \ref{Q2}, \ref{Q3}, and \ref{Q4}.

In Section \ref{sec4}, first, we prove that for every quasitopological group $G$, one has
\begin{align*}
 G\text{ is }(\R,\FCSG)\text{-factorizable}
 &\ \Longleftrightarrow\
 G\text{ is }(\R,\FCQG)\text{-factorizable},\\
 G\text{ is }(\R,\SCSG)\text{-factorizable}
 &\ \Longleftrightarrow\
 G\text{ is }(\R,\SCQG)\text{-factorizable}.
\end{align*} Thus, Questions \ref{Q5} and \ref{Q6} are answered affirmatively, without any complete
regularity assumption on the quasitopological group $G$. The proof uses the
symmetrization construction already considered in \cite[Section 3]{Tk2026}. Next, we show that every countable semitopological group admits
simultaneous factorizations of countable families through second-countable
Tychonoff semitopological groups (see Proposition \ref{prop:countable-orbits}). Applying this result, we obtain that every countable semitopological group is
$(\R,\SCSG_t)$-factorizable and every countable quasitopological group
is $(\R,\SCQG_t)$-factorizable. The same assertions hold for
continuous maps to arbitrary separable metrizable spaces (see Theorem \ref{cor:countable}).

No separation axiom is assumed for any space in this paper unless explicitly stated; for definitions not given here, we refer the reader to \cite{AT,En}. The symbols $\mathbb{R}$, $\mathbb{Q}$, and $\mathbb{N}$ denote the real numbers with the Euclidean topology, the rational numbers, and the positive integers, respectively. For a topological space $X$, $\chi(X)$ denotes its character, $w(X)$ its weight, and $nw(X)$ its network weight. Also, $C(X,Y)$ denotes the set of all continuous functions from $X$ to $Y$.
\section{A countable quasitopological group with non-first-countable reflections}\label{sec2}
We first specify the Tychonoff reflection used in the construction.

\begin{lemma}\label{lem:weak-topology}
Let $(H,\tau)$ be a semitopological group whose continuous real-valued
functions separate points. Let $\tau_c$ be the initial topology on $H$
generated by $C((H,\tau),\R)$. Then $(H,\tau_c)$ is the Tychonoff reflection
of $(H,\tau)$, with the identity as its canonical homomorphism. It is
quasitopological whenever $(H,\tau)$ is quasitopological. If $H$ is
Lindel\"of, $(H,\tau_c)$ is also its regular reflection.
\end{lemma}

\begin{proof}
Let $C((H,\tau),\R)$ denote the set of all continuous real-valued functions on $(H, \tau)$. Consider the evaluation map
$$
e: H \to \mathbb{R}^{C((H,\tau),\R)},
$$
defined by $e(x) = (f(x))_{f \in C((H,\tau),\R)}$ for each $x\in H$.
Since the continuous real-valued functions $C((H,\tau),\R)$ separate points, the map $e$ is injective. The initial topology on $H$ generated by $C((H,\tau),\R)$ is denoted by $\tau_c$. Then one can easily show that $
e: (H,\tau_c)\to \mathbb{R}^{C((H,\tau),\R)}$ is a topological embedding, which implies that $(H, \tau_c)$ is a Tychonoff space and $\tau_c\subseteq \tau$.

Next, we show that $(H, \tau_c)$ is a semitopological group. Let $L_g: H \to H$ denote the left translations by $g \in H$. Take any subbasic element $f^{-1}(U)$ in $(H,\tau_c)$, where $f\in C((H,\tau),\R) $ and $U$ is open in $\R$.
Then $L_g^{-1}(f^{-1}(U))=(f\circ L_g)^{-1}(U)$. Since $(H,\tau)$ is semitopological group, $f\circ L_g$ is a continuous function on $(H,\tau)$. Thus, $(f\circ L_g)^{-1}(U)$ is open in $(H, \tau_c)$, because $(H, \tau_c)$ is generated by $C((H,\tau),\R)$. This shows that $L_g$ is continuous on $(H, \tau_c)$. Similarly, one can show that every right translations $R_g$ by $g \in H$ is continuous on $(H, \tau_c)$. Thus, $(H, \tau_c)$ is a semitopological group. If inversion is
$\tau$-continuous, then one can show that inversion is
$\tau_c$-continuous by the same way.

Now we verify the universal property for the Tychonoff reflection. Let $f:(H,\tau)\to X$ be continuous, where $X$ is Tychonoff. Embed
$X$ into a product of real lines, say $X \subseteq \mathbb{R}^I$ for some index set $I$. For each coordinate projection $\pi_\alpha: \mathbb{R}^I \to \mathbb{R}$, the composition $\pi_\alpha \circ f$ is a continuous real-valued function on $(H, \tau)$, hence $\pi_\alpha \circ f \in C((H,\tau),\R)$. By the definition of $\tau_c$, each $\pi_\alpha \circ f$ is continuous with respect to $\tau_c$. The identity map $id_H: (H, \tau) \to (H, \tau_c)$ is continuous because $\tau_c \subseteq \tau$, and we have $f = f \circ id_H$. This proves that $(H, \tau_c)$ is the Tychonoff reflection of $(H, \tau)$.

Finally, suppose $H$ is Lindel\"of. Let $X$ be a regular space and $f: (H, \tau) \to X$ a continuous map. The subspace $f(H)$ is
Lindel\"of and regular and hence is normal, so $f(H)$
is Tychonoff. Applying the Tychonoff reflection property to the corestriction $f: (H, \tau) \to f(H)$, we obtain a continuous map $h: (H, \tau_c) \to f(H)$ such that $f = h \circ id_H$. Therefore, $(H, \tau_c)$ is also the regular reflection of $(H, \tau)$.
\end{proof}

Let $E=\Q^2$, written additively, and let $o=(0,0)$. Write $\rho$ for
its Euclidean subspace topology, and $\|\cdot\|$ for the Euclidean norm, i.e., $\|x\|=\sqrt{x_1^2+x_2^2}$, where $x=(x_1,x_2)\in \R^2$.
For $n\in\N=\{1,2,\ldots\}$, put
\begin{equation}\label{eq:cones}
 U_n=\{o\}\cup
       \{(s,t)\in\Q^2: |t|<|s|<1/n\}.
\end{equation}
Put
\begin{equation}\label{eq:cones1}
\tau=\{U\subseteq E: \text{~for each~} a\in U \text{~there is ~}n\in\N \text{~such that~ } a+U_n\subseteq U\}.
\end{equation}
Then one can easily show that
\begin{equation}\label{eq:cones2}
E \text{~is endowed the topology~} \tau \text{~is a quasitopological group, denote~} (E,\tau) \text{~by~} G,
\end{equation}
 Since $E$ is countable, the
displayed base is countable. Also $\rho\subseteq\tau$, so $G$ is
Hausdorff and the two coordinate functions separate its points.

For use below, define
\begin{equation}\label{eq:K-A}
 K=\{(s,t)\in E:|t|\leq |s|\},\qquad
 A=E\setminus K.
\end{equation}
The set $K$ is $\rho$-closed, and $U_n\subseteq K$ for every $n$.
The following boundary observation is central:
\begin{equation}\label{eq:boundary}
 (r,-r),(r,r)\in\cl_\tau U_n
       \quad\text{whenever }r\in\Q\text{ and }0\leq |r|<1/n.
\end{equation}
Indeed, since $(0,0) = o \in U_n$, it trivially follows that $(0,0) \in \cl_\tau U_n$.
(i) For $(r,r)$ and $(r,-r)$ when $0 < r < 1/n$: Fix $m \in \mathbb{N}$. Choose a rational $d > 0$ such that $d < \min\{1/m, 1/n - r\}$. Since $0 < d < 1/m$, we have $(d,0) \in U_m$. Consider the point $(r,r) + (d,0) = (r+d, r)$ and $(r,-r) + (d,0) = (r+d, -r)$. Because $d < 1/n - r$, we have $|r+d| = r+d < 1/n$. Also, since $d > 0$, we have $|-r| = |r| < |r+d|< 1/n $, which implies $(r+d, r), (r+d, -r)\in U_n$ by (\ref{eq:cones}). Therefore, both $(r,r) + U_m$ and $(r,-r) + U_m$ meet $U_n$.
(ii) For $(r,r)$ and $(r,-r)$ when $-1/n < r <0$: Fix $m \in \mathbb{N}$. Choose a rational $d > 0$ such that $d < \min\{1/m, 1/n + r\}$. Since $0 < d < 1/m$, we have $(-d,0) \in U_m$. Consider the point $(r,r) + (-d,0) = (r-d, r)$ and $(r,-r) + (-d,0) = (r-d, -r)$. Because $d > 0$ and $r < 0$, we have $ |r-d| = |r| + d.$
Thus, $ |r| < |r| + d = |r-d|$ and $ |-r| = |r| < |r| + d = |r-d|$.
Also, since $d < 1/n + r = 1/n - |r|$, we have
$
|r-d| = |r| + d < |r| + (1/n - |r|) = 1/n.
$
Therefore, $|r|=|-r| < |r-d| < 1/n$ holds, which implies $(r-d, r), (r-d, -r) \in U_n$ by (\ref{eq:cones}). Hence, both $(r,r) + U_m$ and $(r,-r) + U_m$ meet $U_n$.

A topological space is semiregular if regular open sets form a base.
\begin{lemma}\label{lem:semiregular}
The quasitopological group $G$ defined by (\ref{eq:cones2}) is countable, second-countable, Hausdorff and semiregular but is not regular.
\end{lemma}

\begin{proof}
$G$ is first-countable and countable, $G$ is second-countable. Since the topology $\tau$ of $G$ is finer than the Euclidean subspace topology $\rho$, $G$ is Hausdorff and
\[
 \cl_\tau U_n\subseteq\cl_\rho U_n
 =\{(s,t)\in E:|t|\leq |s|\leq1/n\}=:K_n.
\]
A point of $K_n\setminus U_n$ has nonzero first coordinate. If
$|s|=1/n$, arbitrarily small horizontal displacements away from the
vertical axis leave $K_n$. If $|t|=|s|<1/n$, horizontal displacements
toward that axis leave $K_n$. Such displacements belong to every
sufficiently large $U_m$. Hence $\Int_\tau K_n=U_n$, and $U_n$ is
regular open. Translates give semiregularity.

By~\eqref{eq:boundary}, the closure of every $U_n$ contains a point
$(r,r)\ne o$, which is outside $U_1$. Thus no neighborhood of $o$ has
closure contained in $U_1$, and $G$ is not regular.
\end{proof}

Let
\begin{equation}\label{eq:Y-E}
                  Y=(E,\tau_c).
\end{equation}
By Lemma~\ref{lem:weak-topology}, $Y$ is simultaneously $\Tych(G)$ and
$\Reg(G)$, and the canonical map $q:G\to Y$ is the identity on $E$.
The coordinate functions are $\tau_c$-continuous, so
\begin{equation}\label{eq:topology-order}
                  \rho\subseteq\tau_c\subseteq\tau.
\end{equation}

\begin{lemma}\label{lem:new-closure}
The point $o$ belongs to $\cl_Y A$, where $A$ and $Y$ are defined by (\ref{eq:K-A}) and (\ref{eq:Y-E}), respectively.
\end{lemma}

\begin{proof}
Let $O$ be an open neighborhood of $o$ in $Y$. By regularity, choose
an open neighborhood $V$ of $o$ such that $\cl_Y V\subseteq O$.
Since $q$ is continuous, some $U_n$ is contained in $V$. Continuity
also gives
\[
       \cl_\tau U_n\subseteq\cl_Y U_n
                     \subseteq\cl_Y V\subseteq O.
\]
Choose a rational $r$ with $0<r<1/n$. By~\eqref{eq:boundary}, the
point $z=(r,r)$ belongs to $O$. The set $O$ is $\tau$-open, so
$z+U_m\subseteq O$ for some $m$. Take a rational $d$ with
$0<d<\min\{r,1/m\}$. Then
\[
       a=(r-d,r)=z+(-d,0)\in O\cap A.
\]
Every neighborhood $O$ therefore meets $A$.
\end{proof}

\begin{lemma}\label{lem:bump}
Suppose $(a_j)_{j\geq1}$ is a sequence in $A$ defined by (\ref{eq:K-A}) such that
\[
 0<\|a_{j+1}\|<\tfrac14\|a_j\|\quad(j\geq1).
\]
There exists a continuous function $f:G\to[0,1]$ such that $f|_K=0$ and
$f(a_j)=1$ for every $j\geq1$, where $G$ and $K$ are defined by (\ref{eq:cones2}) and (\ref{eq:K-A}), respectively.
\end{lemma}

\begin{proof}
Since $A$ is open in $E$ with Euclidean topology $\rho$, for each $a_j$ we can choose $r_j>0$ so that the Euclidean ball $B_{r_j}^{a_j}=\{x\in E:\|x-a_j\|<r_j \}$ is contained in $A$ and $r_j<\|a_j\|/4$.
Since
\begin{align*}
3\|a_j\|/4 &= \bigl| \|a_j\|/4 - \|a_j\| \bigr| \\
           &< \bigl| \|x-a_j\| - \|a_j\| \bigr| \\
           &\leq \|x\| \\
           &\leq \|x-a_j\| + \|a_j\| \\
           &< \|a_j\|/4 + \|a_j\| \\
           &= 5\|a_j\|/4,
\end{align*}
 the balls $\{B_{r_j}^{a_j}:j\geq 1\}$ are pairwise
disjoint and they form a locally finite family in $(E\setminus\{o\},\rho)$.
For each $j\geq 1$, define $b_j:E\setminus\{o\}\to [0,1]$ by
$$
 b_j(x)=\max\{0,1-\|x-a_j\|/r_j\},
$$
 for each $x\in E\setminus\{o\}.$ Clear, the family $\{b_j: j\geq 1\}$ of continuous functions is locally finite, so $f:E\to [0,1]$ defined by
 $$f(x)=\sum_{j=1}^\infty b_j(x)$$
 for each $x\in E\setminus\{o\}$ and $f(o)=0$ is $\rho$-continuous at each point $x\in E\setminus\{o\}$ and hence $f$ is $\tau$-continuous at each point $x\in E\setminus\{o\}$ by $\rho\subseteq \tau$.
 $f$ vanishes on $K$ defined by (\ref{eq:K-A}), and in particular on the $\tau$-neighborhood
$U_1$ of $o$, so it is also $\tau$-continuous at $o$. Thus $f:G\to[0,1]$ is continuous. Finally,
$f(a_j)=1$ for every $j\geq1$ and $f|_K=0$ by construction.
\end{proof}

Recall that a topological space $X$ is called a \emph{Fr\'{e}chet--Urysohn space} if for every subset $A \subseteq X$ and every point $x \in \cl_X A$, there exists a sequence $(x_n)_{n \in \mathbb{N}}$ in $A$ such that $x_n \to x$.

\begin{proposition}\label{prop:not-frechet}
 $Y=(E,\tau_c)$ defined by (\ref{eq:Y-E}) is not a
Fr\'echet--Urysohn space. Furthermore, $Y$ is not first countable.
\end{proposition}

\begin{proof}
By Lemma \ref{lem:new-closure}, the point $o$ belongs to $\cl_Y A$, where $A$ is defined by (\ref{eq:K-A}). Suppose a sequence in $A$ converges to $o$ in $Y$. By
\eqref{eq:topology-order}, it also converges to $o$ in the Euclidean
topology. It therefore has a subsequence $(a_j)$ satisfying the
hypothesis of Lemma~\ref{lem:bump}. Let $f$ be the function provided
by that lemma. The definition of $\tau_c$ makes $f$ continuous on
$Y$, so convergence would imply $f(a_j)\to f(o)=0$. This contradicts
$f(a_j)=1$. Thus, we have proved that $Y$ is not a
Fr\'echet--Urysohn space.
\end{proof}

\begin{theorem}\label{thm:main-example}
There exists a countable, second-countable, Hausdorff, semiregular
quasitopological abelian group $G$ whose continuous real-valued
functions separate points such that
\begin{enumerate}[label=\textup{(\roman*)},leftmargin=*]
\item the canonical maps to $\Reg(G)$ and $\Tych(G)$ are bijective;
\item $\Reg(G)$ and $\Tych(G)$ are canonically topologically isomorphic;
\item $\Reg(G)$ and $\Tych(G)$ are countable, zero-dimensional
quasitopological groups that are not Fr\'echet--Urysohn.
\end{enumerate}
Consequently, neither first countability nor second countability is
preserved by the regular or Tychonoff reflection, even for countable
Hausdorff quasitopological abelian groups.
\end{theorem}

\begin{proof}
Let $G$ be defined by (\ref{eq:cones2}). Then $G$ is a countable, second-countable, Hausdorff, semiregular
quasitopological abelian group by Lemma \ref{lem:semiregular}. Since the topology $\tau$ of $G$ is finer than the Euclidean subspace topology $\rho$, continuous real-valued
functions on $G$ separate points.

 By Lemma \ref{lem:weak-topology}, $Y=(E,\tau_c)$ defined by (\ref{eq:Y-E}) is the $\Tych(G)$ and $\Reg(G)$. Clearly, the canonical maps to $\Reg(G)$ and $\Tych(G)$ are bijective. Thus,
(i) and (ii) are proved.

For (iii). From Proposition \ref{prop:not-frechet} and (ii) it follows that $\Reg(G)$ and $\Tych(G)$ are not Fr\'echet--Urysohn.
 Since every countable Tychonoff space is zero-dimensional, $\Tych(G)$ is zero-dimensional Hence, so is $\Reg(G)$ by (ii).
\end{proof}

\begin{remark}\label{rem:boundary-sequence}
The boundary sequence $z_j=(1/j,1/j)$ does converge to $o$ in $Y$.
Indeed, the proof of Lemma~\ref{lem:new-closure} shows that every
$Y$-neighborhood of $o$ contains $\cl_\tau U_n$ for some $n$, and
\eqref{eq:boundary} contains the tail of $(z_j)$ in this closure.
In $G$ the same sequence never enters $U_1$. The obstruction in
Proposition~\ref{prop:not-frechet} is different: points strictly
outside $K$ cannot approach $o$ by any sequence in the reflected topology.
\end{remark}

For completeness, the example also rules out a commonly useful mapping
property of the canonical homomorphism. A continuous surjection
$p:X\to Z$ is called \emph{$d$-open} if
$p(U)\subseteq\Int_Z\cl_Z p(U)$ for every open set $U\subseteq X$.

\begin{lemma}\label{lem:d-open}
Let $p:X\to Z$ be a continuous $d$-open surjection, with $Z$ regular.
For $x\in X$,
\[
                   \chi(p(x),Z)\leq\chi(x,X).
\]
Also $w(Z)\leq w(X)$, with cardinal invariants taken to be infinite.
\end{lemma}

\begin{proof}
For a local base $\mathcal B$ of open neighborhoods of $x$, put
$V_B=\Int_Z\cl_Z p(B)$ for $B\in\mathcal B$. These sets contain
$p(x)$ by $d$-openness. Given a neighborhood $O$ of $p(x)$, choose
an open $W$ with $p(x)\in W\subseteq\cl_Z W\subseteq O$, and then
$B\in\mathcal B$ with $B\subseteq p^{-1}(W)$. It follows that
$V_B\subseteq\cl_Z W\subseteq O$, proving the character bound.
The same argument, using a base for $X$ and choosing a preimage of
each point of $Z$, proves the weight bound.
\end{proof}

\begin{corollary}\label{cor:nonquotient}
In Theorem~\ref{thm:main-example}, neither canonical reflection map
is quotient, open, or $d$-open.
\end{corollary}

\begin{proof}
A bijective quotient map is a homeomorphism, whereas $G$ is
second-countable and $Y$ is not first-countable. An open continuous
surjection is quotient. Finally, $d$-openness would contradict
Lemma~\ref{lem:d-open}.
\end{proof}

\begin{remark}\label{rem:semireg-reflection}
The example explicitly separates semiregularization from regular
reflection: $G_{\mathrm{sr}}=G$, but $G$ is not regular. Thus one cannot
identify $T_3(H)$ with $H_{\mathrm{sr}}$ for arbitrary semitopological,
or even quasitopological, groups, where $T_3$ denotes point--closed-set
separation. In particular, this identification, used in the
semitopological part of the proof of~\cite[Theorem~4.2]{PW}, does not
justify the countability assertion made there.
\end{remark}
\section{A second-countable quasitopogical group with factorization obstructions}\label{sec3}
The following extension lemma is a parameter version of the usual
Dugundji construction~\cite{Dugundji}. We give the details because the
parameter space $X$ need not be regular, and no normality of the
product $X\times B$ is available.

\begin{lemma}\label{lem:parameter-extension}
Let $X$ be any topological space, $B$ a metrizable space, and $P$ a
nonempty closed subspace of $B$. Every continuous function
$E:X\times P\to[0,1]$ extends to a continuous function
$F:X\times B\to[0,1]$.
\end{lemma}

\begin{proof}
Fix a compatible metric $d$ on $B$. If $P=B$, there is nothing to
prove. For $b\in B\setminus P$, put $d_b=d(b,P)>0$. By
paracompactness of the metrizable space $B\setminus P$, choose a
locally finite partition of unity $(\lambda_i)_{i\in I}$ subordinate
to an open refinement of the cover
\[
       \{B_d(b,\frac{d_b}{4}):b\in B\setminus P\}.
\]
For each $i$, choose $b_i\in B\setminus P$ so that
$\{b:\lambda_i(b)>0\}\subseteq B_d(b_i,\frac{d_{b_i}}{4})$, and choose
$p_i\in P$ with $d(b_i,p_i)<2d_{b_i}$. Define $F:X\times B\to \R$ by
\[
 F(x,b)=
 \begin{cases}
 E(x,b),&b\in P,\\
 \displaystyle\sum_{i\in I}\lambda_i(b)E(x,p_i),&b\notin P
 \end{cases}
\]
for each $(x,b)\in X\times B.$

Since $(\lambda_i)_{i \in I}$ is a partition of unity, for every $b \in B \setminus P$, we have $\lambda_i(b) \ge 0$ for all $i \in I$ and $\sum_{i \in I} \lambda_i(b) = 1$. Also, since $E: X \times P \to [0,1]$, we have $E(x, p_i) \in [0,1]$ for all $x \in X$ and $p_i \in P$. Therefore, for $b \notin P$,
\[
 0 \le F(x,b) = \sum_{i \in I} \lambda_i(b) E(x, p_i) \le \sum_{i \in I} \lambda_i(b) \cdot 1 = 1.
\]
Thus, $F(x,b) \in [0,1]$ for all $(x,b) \in X \times (B \setminus P)$.

Fix an arbitrary point $(x_0, b_0) \in X \times (B \setminus P)$. Since $P$ is closed in $B$ and by the local finiteness of the partition of unity $(\lambda_i)_{i \in I}$, there exists an open neighborhood $V$ of $b_0$ contained in $B \setminus P$ such that only finitely many indices $i \in I_0 = \{i_1, \dots, i_n\}\subseteq I$ satisfy $\lambda_i|_V \neq 0$.

For any $(x,b) \in X \times V$, the defining sum reduces to a finite sum:
\[
 F(x,b) = \sum_{i \in I_0} \lambda_i(b) E(x, p_i).
\]
For each $i \in I_0$, the function $b \mapsto \lambda_i(b)$ is continuous on $B \setminus P$. Also, because $E$ is jointly continuous on $X \times P$ and $p_i \in P$ is fixed, the function $x \mapsto E(x, p_i)$ is continuous on $X$. Consequently, each term $\lambda_i(b) E(x, p_i)$ is jointly continuous on $X \times V$. Since $I_0$ is finite, the finite sum $F(x,b)$ is jointly continuous at $(x_0, b_0)$. As $(x_0, b_0)$ was arbitrary, $F$ is jointly continuous on $X \times (B \setminus P)$. Thus, $F$ is continuous at each point $(x_0, b_0) \in X \times (B \setminus P)$, because  $X \times (B \setminus P)$ is open in $X\times B$.

Fix $(x_0,p_0)\in X\times P$ and $\eps>0$.
Continuity of $E$
provides a neighborhood $U$ of $x_0$ and $\delta>0$ such that
\begin{equation}\label{eq:d}
 |E(x,p)-E(x_0,p_0)|<\eps
 \quad \text{for all } (x,p) \in U \times V,
\end{equation}\label{eq:Y}
where $V=\{b\in P: d(b,p_0)<\delta\}$
We have the following claim

\medskip
{\bf Claim:} $|F(x,b) - F(x_0, p_0)| < \epsilon$ whenever $(x,b) \in U \times V'$, where $V'=\{b\in B: d(b,p_0)<\delta/4\}$.
\medskip

\noindent \textit{Proof of Claim.} Take any \((x,b)\in U \times V'\). If \(b \in P\), then \(b \in V\) and the inequality follows directly from \eqref{eq:d}. Now suppose $b\notin P$ and $\lambda_i(b)>0$ for some $i\in I$. Then $d(b,b_i)<\frac{d_{b_i}}{4}$ and $$d(b,P)\geq d(b_i,P)-d(b,b_i)>d_{b_i}-\frac{d_{b_i}}{4}=\tfrac34d_{b_i}$$
and
 $$d(b,p_i)\leq d(b,b_i)+d(b_i,p_i)<\frac{d_{b_i}}{4}+2d_{b_i} =\tfrac94d_{b_i}<3d(b,P).$$
Since $p_0\in P$, we have $d(b,P)\leq d(b,p_0)$.
Thus, by $d(b,p_0)<\frac{\delta}{4}$,
$$d(p_i,p_0)\leq d(p_i,b)+d(b,p_0)<3d(b,P)+d(b,p_0)\leq4d(b,p_0)<\delta.$$
Therefore, $|E(x, p_i) - E(x_0, p_0)|<\eps$ by (\ref{eq:d}) and hence,
\begin{align*}
|F(x, b) - E(x_0, p_0)| &= \left| \sum_{i \in I} \lambda_i(b) (E(x, p_i) - E(x_0, p_0)) \right|\\
& \le \sum_{i \in I} \lambda_i(b) |E(x, p_i) - E(x_0, p_0)| \\
&< \sum_{i \in I} \lambda_i(b) \eps = \eps.
\end{align*}
This proves the Claim. Since \((x_0,p_0)\in X\times P\) was arbitrary, \(F\) is continuous at each point of \(X\times P\). Combined with the continuity on \(X\times(B\setminus P)\), we conclude that \(F\) is continuous on \(X\times B\).

\end{proof}

\begin{lemma}\label{lem:universal-function}
For every nonempty countable first-countable space $X$, there are a
Polish abelian topological group $B$ and $F\in C(X\times B,[0,1])$
such that
\begin{equation}\label{eq:all-slices}
          \{F(\,\cdot\,,b):b\in B\}=C(X,[0,1]).
\end{equation}
One may take $B=\R^X\times\mathbb Z^{X\times\N}$ with its product
group topology, where $\mathbb Z$ is discrete.
\end{lemma}

\begin{proof}
Choose a countable open local base $(U_n(x))_{n\in\N}$ at each
$x\in X$. Inside the Polish group $B$ in the statement, consider
\begin{equation}\label{eq:closed-parameters}
 \begin{split}
 P=\bigl\{(v,m)\in[0,1]^X\times\N^{X\times\N}:\quad
 &|v(y)-v(x)|\leq 2^{-k}\\[-2pt]
 &\text{for all }x\in X,\ k\in\N,\ y\in U_{m(x,k)}(x)\bigr\}.
 \end{split}
\end{equation}
Now we shall prove that $P$ is closed in $B=\R^X\times\mathbb Z^{X\times\N}$. Let $(v_0,m_0) \in P^c= B \setminus P$. First, if $(v_0,m_0) \notin [0,1]^X \times \N^{X \times \N}$, then it can be separated from $P$ by an obvious basic open set. Thus, we may assume without loss of generality that $(v_0,m_0) \in [0,1]^X \times \N^{X \times \N}$. Since $(v_0,m_0) \notin P$, the defining inequality must fail. Hence, there exist $x_0 \in X$, $k_0 \in \N$, and $y_0 \in X$ such that $y_0 \in U_{m_0(x_0,k_0)}(x_0)$ and
$$
|v_0(y_0) - v_0(x_0)| > 2^{-k_0}.
$$

Set $n_0 = m_0(x_0,k_0) \in \mathbb{N}$. Define the sets $W_1, W_2 \subseteq B$ as follows:
\[
W_1 = \bigl\{(v', m') \in B : m'(x_0,k_0) = n_0\bigr\}=\pi_{x_0,k_0}^{-1}(n_0),
\]
where $\pi_{x_0,k_0} : B \to \mathbb{Z}$ is the projection map, defined by $\pi_{x_0,k_0}(v', m') = m'(x_0,k_0)$.
\[
W_2 = \bigl\{(v', m') \in B : |v'(y_0) - v'(x_0)| > 2^{-k_0}\bigr\}.
\]
Let $W = W_1 \cap W_2$. Clearly, $(v_0,m_0) \in W$.

We claim that $W$ is open in $B$. Since $\{n_0\}$ is open in $\mathbb{Z}$ and the  projection map $\pi_{x_0,k_0}$ is continuous, $W_1$ is open in $B$. For $W_2$: Consider the evaluation map $\phi : B \to \mathbb{R}^2$ defined by $\phi(v', m') = (v'(y_0), v'(x_0))$. Let $\pi_i:\R^2\to \R$ be the $i$-th projection ($i=1,2$) and $\pi_{x}:B\to \R$ be the $x$-th projection ($x=y_0,x_0$). Then one can easily show that $\pi_1\circ \phi = \pi_{y_0}$, $\pi_2\circ \phi = \pi_{x_0}$. Since $\pi_{y_0}$ and $\pi_{x_0}$ are continuous, $\phi$ is continuous. The set $O = \{(a,b) \in \mathbb{R}^2 : |a - b| > 2^{-k_0}\}$ is open in $\mathbb{R}^2$. Therefore, $W_2 = \phi^{-1}(O)$ is open in $B$.
Thus, $W=W_1 \cap W_2$ is open in $B$.

Finally, for any $(v', m') \in W$, we have $m'(x_0,k_0) = n_0$, $y_0 \in U_{n_0}(x_0) = U_{m'(x_0,k_0)}(x_0)$, and $|v'(y_0) - v'(x_0)| > 2^{-k_0}$. This means that $(v', m')$ violates the defining condition of $P$, so $(v', m') \in P^c$. Hence, $W \subseteq P^c$, which means $P^c$ is open. Consequently, $P$ is closed in $B$.

$P$ is nonempty, as witnessed by any constant $v$.

 Define the evaluation function $$E:X\times P\longrightarrow[0,1]$$ by
$E(x,(v,m))=v(x)$ for each $(x,(v,m))\in X\times P$.

To prove $E$ is continuous, consider the map $\Phi: X \times P \to X \times [0,1]^X$ defined by $\Phi(x,(v,m)) = (x,v)$. Since the projections $\pi_X(x,(v,m))=x$ and $\pi_P(x,(v,m))=v$ are continuous, so is $\Phi$. Let $\operatorname{ev}: X \times [0,1]^X \to [0,1]$ be the evaluation map $\operatorname{ev}(x,v)=v(x)$. By the definition of the product topology on $[0,1]^X$, the evaluation map $\operatorname{ev}$ is continuous. Since $E = \operatorname{ev} \circ \Phi$, it follows that $E$ is continuous.

Apply Lemma~\ref{lem:parameter-extension} to extend $E$ to
$F:X\times B\to[0,1]$. Every member of $C(X,[0,1])$ already
occurs as a slice with parameter in $P$, and every slice of $F$
is continuous. This proves~\eqref{eq:all-slices}.
\end{proof}

\begin{lemma}\label{lem:initiality-test}
Let $Y$ be Tychonoff and $a:Y\to L$ continuous, where $L$ is any
topological space. If every $\varphi\in C(Y,[0,1])$ has the form
$\varphi=u_\varphi\circ a$ for some $u_\varphi\in C(L,[0,1])$,
then $a$ is a topological embedding.
\end{lemma}

\begin{proof}
Continuous $[0,1]$-valued functions separate points of $Y$, so $a$
is injective. They also generate the topology of $Y$. Since each
of them is pulled back from a continuous function on $L$, the
topology induced by $a$ contains the topology of $Y$. Continuity
of $a$ gives the reverse inclusion.
\end{proof}

\begin{theorem}\label{thm:amplification}
Let $G_0$ be a countable first-countable semitopological group whose
continuous real-valued functions separate points, and suppose that
$Y_0=\Tych(G_0)$ is not first-countable. Let $B$ and $F$ be provided
by Lemma~\ref{lem:universal-function} for the underlying space of
$G_0$, and put $H=G_0\times B$. Then:
\begin{enumerate}[label=\textup{(\roman*)},leftmargin=*]
\item $H$ is second-countable, and its regular and Tychonoff
reflections coincide with a Tychonoff semitopological group $Z$
having a countable network. The canonical map $r:H\to Z$ is
bijective.
\item $F$ cannot be written as $h\circ p$, with $p:H\to L$ a
continuous homomorphism to a regular first-countable semitopological
group and $h\in C(L,[0,1])$.
\item Writing $F=\bar F\circ r$, the function $\bar F$ cannot be
written as $h\circ p$ with $p:Z\to L$ a continuous homomorphism
to any first-countable semitopological group and $h\in C(L,[0,1])$.
\end{enumerate}
If $G_0$ is quasitopological or abelian, then so are $H$ and $Z$.
If $G_0$ is semiregular, then so is $H$.
\end{theorem}

\begin{proof}
A countable space with countable local bases has a countable base.
Thus $H$ is second-countable and Lindel\"of. Its continuous real-valued
functions separate points, using those on $G_0$ and $B$ separately.
From Lemma~\ref{lem:weak-topology} it follows that the regular and Tychonoff
reflections of $H$ coincide, denote by $Z$, with the
initial topology generated by $C(H,\R)$ on the same underlying
group. The images under $r$ of a countable base of $H$ form a
network for $Z$. This proves~\textup{(i)}. Let $q:G_0\to Y_0$ be
the canonical map; by the same lemma it is bijective and $Y_0$
is also $\Reg(G_0)$.

For~\textup{(ii)}, suppose $F=h\circ p$ as stated and write $e$
for the identity of $G_0$ and $0$ for that of $B$. Since $p:G_0\times B\rightarrow L$ is continuous, we can define a continuous
map $p_0:G_0\rightarrow L$ by $p_0(x)= p(x,0)$ for each $x\in G_0$. By the
universal property of $q$, it has the form $a\circ q$ for a
continuous map $a:Y_0\to L$. For every
$\varphi\in C(Y_0,[0,1])$, we have $\varphi\circ q \in C(G_0,[0,1])$. Hence,  by~\eqref{eq:all-slices}, we can choose $b_\varphi\in B$ such that
\[
       F(x,b_\varphi)=\varphi(q(x))
\]
holds for each $x\in G_0$. Put
$c_\varphi=p(e,b_\varphi)$. Since $F=h\circ p$ and $p$ is a homomorphism, we obtain
\begin{align*}
\varphi(q(x))&=F(x,b_\varphi)\\
  & =h\bigl(p((x,0)(e,b_\varphi))\bigr),\quad \text{by~} (x,b_\varphi)=(x,0)(e,b_\varphi)\\
  & =h\bigl(p(x,0)p(e,b_\varphi)\bigr)\\
  & =h\bigl(p_0(x)c_\varphi\bigr)\\
   &=h\bigl(a(q(x))c_\varphi\bigr), \quad \text{by~} p_0=a\circ q.\\
\end{align*}
Since $q$ is onto, $\varphi=u_\varphi\circ a$, where $u_\varphi:L\to [0,1]$ defined by
$u_\varphi(t)=h(tc_\varphi)$ is continuous by separate continuity
in $L$. Thus, we have the following diagram:

\[
\xymatrix@R=2em@C=2.5em{
  G_0 \ar[r]^{p_0} \ar[d]_q &L \ar[ddl]^{u_\varphi} \\
  Y_0 \ar[ur]^a \ar[d]_\varphi & \\
  [0,1] &
}
\]
Lemma~\ref{lem:initiality-test} makes $a$ an embedding
of $Y_0$ into the first-countable space $L$, a contradiction with $Y_0$ not first countable.

For~\textup{(iii)}, the universal property of $r:H\to Z$ gives
$\bar F\in C(Z,[0,1])$ with $F=\bar F\circ r$. Then we can define a continuous
map $r_0:G_0\to Z$ by $r_0(x)= r(x,0)$. The universal property of $q:G_0\to Y_0$ gives a continuous map $s:Y_0\to Z$ satisfying
\[
                  s(q(x))=r_0(x)
\]
holds for each $x\in G_0$.
Suppose $\bar F=h\circ p$ through a first-countable semitopological
group $L$.
\[
\xymatrix@R=2em@C=2.5em{
  H\ar[r]^{F} \ar[d]_r &[0,1]  \\
  Z \ar[ur]^{\overline{F}} \ar[d]_p & \\
  L \ar[uur]^{h}&
}
\]
Put $a=p\circ s$. For $\varphi \in C(Y_0,[0,1])$, we have $\varphi\circ q\in C(G_0,[0,1])$. Hence, by Lemma \ref{lem:universal-function}(\ref{eq:all-slices}), there
$b_\varphi \in B$ such that $F(x,b_\varphi)=\varphi(q(x))$ holds for each $x\in G_0$, put $c_\varphi=p(r(e,b_\varphi))$. Because $r$ and $p$
are homomorphisms, the same calculation gives
\begin{align*}
 \varphi(q(x))&=F(x,b_\varphi)\\
 &=h(p(r(x,b_\varphi)))\\
 &=h(p(r(x,0)r(e,b_\varphi)))\\
  &=h(p(r(x,0))p(r(e,b_\varphi)))\\
 &=h(p(r_0(x))c_\varphi)\\
 &=h(a(q(x))c_\varphi).\\
\end{align*}
Since $q$ is onto, $\varphi=u_\varphi\circ a$, where $u_\varphi:L\to [0,1]$ defined by
$u_\varphi(t)=h(tc_\varphi)$ is continuous by separate continuity
in $L$. Thus, we have the following diagram:
\[
\xymatrix{
  G_0 \ar[r]^{r_0} \ar[d]_r & Z \ar[d]^p \\
  Y_0 \ar[ur]^s \ar[r]^{a=p\circ s} \ar[d]_\varphi & L \ar[dl]^{u_\varphi} \\
  {[0,1]} &
}
\]

Lemma~\ref{lem:initiality-test} again gives an embedding
of $Y_0$ into $L$, a contradiction with $Y_0$ not first countable.

The final assertions about the operations follow from products and
Lemma~\ref{lem:weak-topology}. For semiregularity, take products
of regular open basic sets in $G_0$ and $B$: such a product is
regular open, since closure and interior commute with nonempty
finite products.
\end{proof}

\begin{theorem}\label{thm:main-factorization}
There exist a Hausdorff quasitopological abelian group $H$, which is
second-countable and semiregular, and a function $F\in C(H,[0,1])$
with the following properties.
\begin{enumerate}[label=\textup{(\roman*)},leftmargin=*]
\item The regular and Tychonoff reflections of $H$ coincide with a
Tychonoff quasitopological abelian group $Z$ having a countable network.
The canonical map $r:H\to Z$ is
bijective;
\item $F$ does not factor through any continuous homomorphism from
$H$ to a regular first-countable semitopological group;
\item The function $\bar F =\bar F\circ r\in C(Z,[0,1])$ does not
factor through any continuous homomorphism to a first-countable
semitopological group.
\end{enumerate}
In particular, $H$ is $(\R,\SCSG)$-factorizable but not
$(\R,\FCSG_r)$-factorizable, whereas neither $\Reg(H)$ nor
$\Tych(H)$ is $(\R,\FCSG)$-factorizable.
\end{theorem}

\begin{proof}
Let $G$ be the quasitopological group of Theorem~\ref{thm:main-example}. There are a
Polish abelian topological group $B=\R^G\times \Z^{G\times \N}$ and $F\in C(H,[0,1])$ provided
by Lemma~\ref{lem:universal-function} for the underlying space of
$G$, where $H=G\times B$. All the asserted separation
and algebraic properties in Theorem \ref{thm:amplification} now follow. Thus (i)-(iii) follow from (i)-(iii) in Theorem \ref{thm:amplification}, respectively.

Every second-countable
semitopological group is $(\R,\SCSG)$-factorizable, by using its
identity homomorphism for each continuous real-valued function.
The two failures of factorization are witnessed by $F$ and $\bar F$.
\end{proof}

\begin{remark}\label{cor:61-65}
Use $H$ from Theorem~\ref{thm:main-factorization}. Questions~\ref{Q1}, \ref{Q2}, \ref{Q3}, and~\ref{Q4} have negative
answers, even with a second-countable Hausdorff semiregular
quasitopological abelian group as the source. In Questions~\ref{Q1}
and~\ref{Q2} both the regular and the Tychonoff reflection can fail
the asserted factorization property. In Question~\ref{Q4} even regular
first-countable targets are insufficient.
\end{remark}

\begin{remark}\label{rem:PW-counterexample}
The group $H$ from Theorem~\ref{thm:main-factorization} also provides counterexamples to the semitopological
and quasitopological assertions in both~\cite[Theorem~4.2]{PW}
and~\cite[Proposition~4.20]{PW}. Indeed, $H$ is second-countable and
quasitopological, yet it fails factorization even through regular
first-countable semitopological groups. The failed identification
in the proof of the former result was explained in
Remark~\ref{rem:semireg-reflection}.
\end{remark}


\section{Some positive answers in factorizations}\label{sec4}

We now turn to the positive answers to Questions~\ref{Q5} and~\ref{Q6}. 
For a semitopological group $H=(H,\sigma)$, we consider the semitopological group topology $ \sigma^{-1} = \{U^{-1} : U \in  \sigma\}$ on $H$. Clearly, the groups $(H,  \sigma)$ and $(H,  \sigma^{-1})$ are homeomorphic. Consider the family
\[
\mathcal{B} = \{U \cap U^{-1} : e \in U \in  \sigma\},
\]
where $e$ is the identity of $H$. It is easy to verify that $\mathcal{B}$ is a local base at the identity of $H$ for a quasitopological group topology $\lambda$ on $H$ which is finer than both $ \sigma$ and $ \sigma^{-1}$. The group $(H, \lambda)$ is referred to as the \textit{quasitopological group associated to} $G$ and is denoted by $H^*$ (see \cite{Tk2026}). The identity mapping of $H^*$ onto $H$ is a continuous homomorphism. If $H$ is a paratopological group, then $H^*$ is a topological group (see \cite{BR} or \cite[ Section 4.1]{Tk20142}).

\begin{lemma}~{\cite[Lemma~3.1 and
Proposition~3.4]{Tk2026}}\label{lem:symmetrization}
Let $H$ be a semitopological group.
\begin{enumerate}[label=\textup{(\roman*)},leftmargin=*]
\item $H^*$ is a quasitopological group, and the identity
$j:H^*\to H$ is a continuous homomorphism.
\item Every continuous homomorphism $p:G\to H$ from a
quasitopological group has a continuous lift
$p^*:G\to H^*$ with the same underlying homomorphism.
\item If a topological property $\mathcal P$ is invariant under
homeomorphisms and preserved by subspaces and finite products,
then $H\in\mathcal P$ implies $H^*\in\mathcal P$.
\end{enumerate}
\end{lemma}

Write $\SG(\mathcal P)$ and $\QG(\mathcal P)$ for the semitopological
and quasitopological groups whose underlying spaces have $\mathcal P$.

\begin{theorem}\label{thm:transfer-symmetry}
Let $\mathcal X$ be any class of topological spaces, and let
$\mathcal P$ satisfy Lemma~\ref{lem:symmetrization}\textup{(iii)}.
For every quasitopological group $G$, the following are equivalent:
\begin{enumerate}[label=\textup{(\roman*)},leftmargin=*]
\item $G$ is $(\mathcal X,\SG(\mathcal P))$-factorizable;
\item $G$ is $(\mathcal X,\QG(\mathcal P))$-factorizable.
\end{enumerate}
Moreover, each individual factorization in~\textup{(i)} can be
converted into one in~\textup{(ii)} without changing the underlying
homomorphism, its kernel, or the underlying map to the target space.
\end{theorem}

\begin{proof}
Only \textup{(i)}$\Rightarrow$\textup{(ii)} needs proof. For a
continuous $f:G\to X$, $X\in\mathcal X$, take
$f=h\circ p$ with $p:G\to H$ continuous and onto and
$H\in\SG(\mathcal P)$. By Lemma~\ref{lem:symmetrization},
$p^*:G\to H^*$ is continuous and onto and
$H^*\in\QG(\mathcal P)$. Since the identity $j:H^*\to H$ is continuous,
\[
                      f=(h\circ j)\circ p^*
\]
is the required factorization.
\[
\xymatrix{
  & G \ar[r]^f \ar[d]^p \ar[dl]_{p^*} & X \\
  H^* \ar[r]^j & H \ar[ur]_h &
}
\]
\end{proof}

\begin{corollary}\label{cor:66-67}
For every quasitopological group $G$, without any separation
assumption, one has
\begin{align*}
 G\text{ is }(\R,\FCSG)\text{-factorizable}
 &\ \Longleftrightarrow\
 G\text{ is }(\R,\FCQG)\text{-factorizable},\\
 G\text{ is }(\R,\SCSG)\text{-factorizable}
 &\ \Longleftrightarrow\
 G\text{ is }(\R,\SCQG)\text{-factorizable}.
\end{align*}
Thus Problems~6.6 and~6.7 of~\cite{Tk2026} have affirmative answers.
\end{corollary}

\begin{proof}
First countability and second countability are hereditary and
finitely productive topological properties. Apply
Theorem~\ref{thm:transfer-symmetry}.
\end{proof}

\begin{corollary}\label{cor:cardinal}
The same equivalence holds with first or second countability replaced
by any of the following target conditions: metrizability,
$\chi(H)\leq\kappa$, $w(H)\leq\kappa$, or $\nw(H)\leq\kappa$,
where $\kappa$ is an infinite cardinal. A Hausdorff, regular, or
Tychonoff requirement already imposed on the targets can be retained.
These statements hold for arbitrary $\mathcal X$ in place of $\R$.
\end{corollary}

\begin{proof}
Each listed property, and its conjunction with any listed separation
property, satisfies Lemma~\ref{lem:symmetrization}\textup{(iii)}.
\end{proof}


Proposition \ref{prop:countable-orbits} records a direct
factorization argument for countable groups. It also explains why
the counterexample to countability preservation is compatible with
strong factorization properties.

\begin{proposition}\label{prop:countable-orbits}
Let $H$ be a countable semitopological group and
$\mathcal F\subseteq C(H,\R)$ a countable family. There exist a
second-countable Tychonoff semitopological group $L$ and a continuous
surjective homomorphism $p:H\to L$ such that every member of
$\mathcal F$ factors continuously through $p$. If $H$ is
quasitopological, $L$ can be chosen quasitopological.
\end{proposition}

\begin{proof}
For each function $f:H\to \R$ and $a,b\in H$, define
\begin{equation}\label{ivan}
f_{ab}:H\to \R
\end{equation} by $f_{ab}(x)=f(axb)$ for each $x\in H$.
Since $\mathcal F$ and $H$ are countable, we can obtain a countable family
\[
 \mathcal A=\{f_{ab}: f\in\mathcal F,\ a,b\in H\}.
\]
 Since $H$ is semitopological group and $f$ is continuous, this family consists of continuous functions such that the function $u_{ab}$ is contained in $\mathcal A$ for each $u\in \mathcal A$ and each $a,b\in H$, where $u_{ab}:H\to \R$
 is defined by (\ref{ivan}). Define relation on $H$ by $x\sim y$ when
$u(x)=u(y)$ for every $u\in\mathcal A$. Then translation invariance makes this an equivalence relation. Also, this relation is invariant
under both left and right multiplication. Indeed, if $x\sim y$, then $u(x)=u(y)$ holds for each $u\in\mathcal A$. Hence, $u_{ab}(x)=u_{ab}(y)$ holds for each $a,b\in H$, which implies that $u(axb)=u(ayb)$ by (\ref{ivan}), so $axb\sim ayb$. Thus, the equivalence relation class $N$ containing
the identity $e$ is a normal subgroup: if $n,m\sim e$, then
 translating $m\sim e$ by $n$ gives $nm\sim n\sim e$; translating $n\sim e$ by $n^{-1}$ gives
$n^{-1}\sim e$; and translating $n\sim e$ by $a,a^{-1}$ gives $ana^{-1}\sim e$ for each $a\in H$, which implies that $N$ is a normal subgroup.
Also, we can obtain that the equivalence classes are the cosets of $N$.

Let $p:H\to H/N $ be the natural group homomorphism. Then each $u\in\mathcal A$ induces a function $\overline{u}:H/N\to\R$ such that $u=\overline{u}\circ p$, since $u$ is constant on $xN$ for each $x\in H$.
Put $\overline{\mathcal A}=\{\overline u: u\in\mathcal A\}$. Define the map
\[
 \iota:H/N\longrightarrow\R^{\overline{\mathcal A}}
 \]
 by $\iota(x)=(\overline u(x))_{\overline{u}\in\overline{\mathcal A}}$ for each $x\in H/N$.
Then $ \iota$ is injective. Indeed, if $x_1\neq x_2$ for $x_1,x_2\in H/N$, then there are $h_1,h_2\in H$ such that $x_1=p(h_1), x_2=p(h_2)$ and $h_1\notin h_2N$, which implies that
there is $u\in \mathcal A$ such that $u(h_1)\neq u(h_2)$. Then $\overline{u}(x_1)\neq \overline{u}(x_2)$ by $u=\overline{ u}\circ p$. Hence, $\iota(x_1)\neq \iota(x_2)$. This proves that $ \iota$ is injective. Give $L=H/N$ the topology $\tau_{\overline{\mathcal A}}$ generated
 by $\overline{\mathcal A}$. i.e., the family $\beta=\{\overline{ u}^{-1}(U): \overline {u}\in \overline{\mathcal A}, U \text{~open in ~}\R\}$ is a subbase of the topology $\tau_{\overline{\mathcal A}}$. Then one can easily show that  $\iota:(L, \tau_{\overline{\mathcal A}})\to \R^{\overline{\mathcal A}}$ is a topological embedding. Hence, $(L, \tau_{\overline{\mathcal A}})$ This is a second-countable Tychonoff space.
Also, the algebraic quotient homomorphism $p:H\to (L, \tau_{\overline{\mathcal A}})$ is continuous,
since each coordinate of $\iota\circ p$ belongs to $\mathcal A$.

Next, we shall prove that $(L, \tau_{\overline{\mathcal A}})$ is a semitopological group when so is $H$. For each $a\in L$, denote the right translation on $L$ by $\rho_a$. Take $x\in H$ such that $a=p(x)$ and denote the right translation on $H$ by $\rho_x$. Take any element $\bar u^{-1}(U)$ of the subbase $\beta$ of the topology $\tau_{\overline{\mathcal A}}$. Then, by the following commutative diagram,
\[
\xymatrix{
H \ar[r]^{\rho_x} \ar[d]^p & H \ar[r]^{u} \ar[d]^p & \R \\
L \ar[r]^{\rho_a} & L \ar[ru]_{\overline {u}} &
}
\]
For any subbase element \(\overline{u}^{-1}(U)\), we obtain
\begin{align*}
 \rho_a^{-1}(\bar u^{-1}(U))&=p(\rho_x^{-1}(u^{-1}(U)))\\
 &=p((u\circ \rho_x)^{-1}(U))\\
  &=p(u_{ex}^{-1}(U)) \qquad \text{by~} u_{ex}=\overline{u_{ex}}\circ p\\
  &=\overline{u_{ex}}^{-1}(U)\in \beta  \qquad \text{by~} \overline{u_{ex}}\in \overline{\mathcal A},
\end{align*}
which implies that $\rho_a:(L,\tau_{\bar{\mathcal A}})\to (L,\tau_{\bar{\mathcal A}})$ is continuous for each $a\in L$. Similarly, using $u_{xe}$ we can show that every left translation on $L$ is continuous. Thus, $(L,\tau_{\bar{\mathcal A}})$ is a semitopological group.

If $H$ is quasitopological, replace $L$ by $L^*$ and use
Lemma~\ref{lem:symmetrization}; the target remains second-countable
and Tychonoff.
\end{proof}

\begin{theorem}\label{cor:countable}
Every countable semitopological group is
$(\R,\SCSG_t)$-factorizable. Every countable quasitopological group
is $(\R,\SCQG_t)$-factorizable. The same assertions hold for
continuous maps to arbitrary separable metrizable spaces.
\end{theorem}

\begin{proof}
For a single real-valued function use
Proposition~\ref{prop:countable-orbits}. For a continuous map
$f:H\to X$ with $X$ separable metrizable, embed $X$ into
$\R^\N$ and apply the proposition simultaneously to its coordinate
functions. The induced map on the common homomorphic image is
continuous into $\R^\N$ and takes values in the embedded copy of
$X$, as the homomorphism is onto.
\end{proof}

\begin{corollary}
Let the quasitopological group $G$ be defined by (\ref{eq:cones2}). Then $G$, $\Reg(G)$, and $\Tych(G)$ are all
$(\R,\SCQG_t)$-factorizable.
\end{corollary}

\begin{proof}
Both $G$ and its common reflection $Y$ are countable
quasitopological groups, so it follows from Theorem~\ref{cor:countable}.
\end{proof}

\section*{References}


\def\cprime{$'$}

\end{document}